\documentclass[10pt]{amsart}
\input{amssymb.sty}

\usepackage{amsmath,amsthm}
\usepackage{hyperref}
\usepackage[english]{babel}

\usepackage{mathrsfs,amssymb}

\title{On cocycle superrigidity for Gaussian actions}
\author{Jesse Peterson}
\address{Jesse Peterson, Vanderbilt University, 1326 Stevenson Center, Nashville, TN 37240}
\email{jesse.d.peterson@vanderbilt.edu}
\author{Thomas Sinclair}
\address{Thomas Sinclair, Vanderbilt University, 1326 Stevenson Center, Nashville, TN 37240}
\email{thomas.sinclair@vanderbilt.edu}
\thanks{The first author's research is partially supported by NSF Grant 0901510 and a grant from the Alfred P. Sloan Foundation}

\dedicatory{}
\newtheorem{thm}{Theorem}[section]
\newtheorem{cor}[thm]{Corollary}
\newtheorem{lem}[thm]{Lemma}
\theoremstyle{definition}
\newtheorem{prop}[thm]{Proposition}
\newtheorem{defn}[thm]{Definition}
\newtheorem{rem}[thm]{Remark}

\newcommand{\car}{\curvearrowright}
\newcommand{\dd}{\delta}
\newcommand{\G}{\Gamma}
\newcommand{\g}{\gamma}
\newcommand{\e}{\varepsilon}
\newcommand{\A}{{\mathbb A}}
\newcommand{\C}{{\mathbb C}}
\newcommand{\F}{{\mathbb F}}
\newcommand{\N}{{\mathbb N}}
\newcommand{\Z}{{\mathbb Z}}
\newcommand{\R}{{\mathbb R}}
\newcommand{\T}{{\mathbb T}}

\newcommand{\GG}{{\mathcal G}}
\newcommand{\HH}{{\mathcal H}}
\newcommand{\JJ}{{\mathcal J}}
\newcommand{\KK}{{\mathcal K}}
\newcommand{\LL}{{\mathcal L}}
\newcommand{\OO}{{\mathcal O}}
\newcommand{\UU}{{\mathcal U}}

\newcommand{\SSS}{{\mathfrak S}}
\newcommand{\Ad}{\operatorname{Ad}}

\newcommand{\Aut}{\operatorname{Aut}}
\newcommand{\id}{\operatorname{id}}

\newcommand{\ip}[2]{\langle #1, #2 \rangle}

\begin{document}
\begin{abstract}  We present a general setting to investigate $\UU_{\rm fin}$-cocycle superrigidity for Gaussian actions in terms of closable derivations on von Neumann algebras.  In this setting we give new proofs to some $\UU_{\rm fin}$-cocycle superrigidity results of S. Popa and we produce new examples of this phenomenon.  We also use a result of K. R. Parthasarathy and K. Schmidt to give a necessary cohomological condition on a group representation in order for the resulting Gaussian action to be $\UU_{\rm fin}$-cocycle superrigid.
\end{abstract}

\maketitle

\section*{Introduction}

A central motivating problem in the theory of measure-preserving actions of countable groups on probability spaces is to classify certain actions 
up to orbit equivalence, \textit{i.e.}, isomorphism of the underlying probability spaces such that the orbits of one group are carried onto the orbits of the other.  When the 
groups are amenable this problem was completely settled in the early '80s (cf. \cite{dyeI, dye, ornsteinweiss, connesfeldmanweiss}): 
all free ergodic actions of countable, discrete, amenable groups are orbit equivalent.  The nonamenable case, however, is much more complex and has recently seen 
a flourish of activity including a number of striking results.  We direct the reader to the survey articles \cite{popadefrigid, shalom} for a summary of these recent developments.

One breakthrough which we highlight here is Popa's use of his deformation/rigidity techniques in von Neumann algebras to produce 
rigidity results for orbit equivalence (cf. \cite{popabernoulli, popacohomology, popamalleableI, popamalleableII, popasuperrigid, popaspectralgap, popasasyk}).  
One of the seminal results using these techniques is Popa's Cocycle Superrigidity Theorem \cite{popasuperrigid, popaspectralgap} (see also \cite{furmancsr} 
and \cite{vaes} for more on this) which obtains orbit equivalence superrigidity results by means of untwisting cocycles into a finite von Neumann algebra.  
In order to state this result we recall a few notions regarding groups.

A subgroup $\G_0 < \G$ is \textbf{w-normal} if there is a chain of subgroups $\G_0 = \G_1 < \G_2 < \dots < \G_\beta = \G$ for some ordinal $\beta$ such that $(\bigcup_{\alpha' < \alpha} \G_{\alpha'}) \lhd \G_\alpha$ for all $\alpha \leq \beta$.  A group $\G$ is \textbf{w-rigid} if there exists an infinite w-normal subgroup $\G_0$ such that the pair $(\G, \G_0)$ has relative property (T).
If $\UU$ is a class of Polish groups then a free, ergodic, measure-preserving action of a countable discrete group $\G$ on a standard probability space $(X, \mu)$
is said to be \textbf{$\UU$-cocycle superrigid} if any cocycle for the action $\G \car (X, \mu)$ which is valued in a group contained in the class $\UU$ must 
be cohomologous to a homomorphism.
$\UU_{\mathrm {fin}}$ is used to denote the class of Polish groups which arise as closed subgroups of the unitary groups of $\mathrm {II}_1$ factors. In particular, the class of compact Polish groups and the class of countable discrete groups are both contained in $\UU_{\mathrm {fin}}$. The notions of w-normality, w-rigidity, and the class $\UU_{\mathrm {fin}}$ are due to Popa (cf. \cite{popacohomology, popasuperrigid}).  
\\

\noindent
{\bf Popa's Cocycle Superrigidity Theorem} (\cite{popasuperrigid}, \cite{popaspectralgap}) (for Bernoulli shift actions). Let $\G$ be a group which is either w-rigid or contains a w-normal subgroup which is the direct product of an infinite group and a nonamenable group, and let $(X_0, \mu_0)$ be a standard probability space.  Then the Bernoulli shift action 
$\G \car \Pi_{g \in \G} (X_0,\mu_0)$ is $\UU_{\mathrm {fin}}$-cocycle superrigid. \\

The proof of this theorem uses a combination of deformation/rigidity and intertwining techniques that were initiated in \cite{popabetti}.
Roughly, if we are given a cocycle into a unitary group of a $\mathrm {II}_1$ factor, we may consider the ``twisted''
group algebra sitting inside of the group-measure space construction.  The existence of rigidity can then be contrasted against natural malleable deformations
from the Bernoulli shift in order to locate the ``twisted'' algebra inside of the group-measure space construction.  Locating the ``twisted'' algebra allows us 
to ``untwist'' it, and, in so doing, untwist the cocycle in the process.

The existence of such s-malleable deformations (introduced by Popa in \cite{popamalleableI, popamalleableII}) actually occurs in a broader setting than the (generalized) Bernoulli shifts with diffuse core, but it was Furman \cite{furmancsr} who first noticed that the even larger class of Gaussian actions are also s-malleable.  The class of Gaussian actions has a rich structure, owing to the fact the every Gaussian action of a group $\G$ arises functorially from an orthogonal representation of $\G$.  The interplay between the representation theory and the ergodic theory of a group via the Gaussian action has been fruitfully exploited in the literature (cf. the seminal works of Connes \& Weiss and of Schmidt, 
\cite{connesweiss, schmidtamen, schmidtcohomology}, \textit{inter alios}).

In this paper, we will explore $\UU_{\mathrm {fin}}$-cocycle superrigidity within the class of Gaussian actions.  An advantage to our approach is that we develop a general framework for investigating cocycle superrigidity of such actions by using derivations on von Neumann algebras.  The first theme we take up is the relation between the cohomology of group representations and the cohomology of their respective Gaussian actions.
Under general assumptions, we show that cohomological information coming from the representation can be faithfully transferred to the cohomology group of the action with coefficients in the circle group $\T$.  As a consequence, we obtain our first result, that the cohomology of the representation provides an obstruction to the $\UU_{\mathrm {fin}}$-cocycle superrigidity of the associated Gaussian action.

\begin{thm}
Let $\G$ be a countable discrete group and $\pi: \G \rightarrow \OO(\KK)$ a weakly mixing orthogonal representation.
A necessary condition for the corresponding Gaussian action to be $\{ \T \}$-cocycle superrigid is for   
$H^1(\G, \pi) = \{ 0 \}$.
\label{thm:1}
\end{thm}

The Bernoulli shift action of a group is precisely the Gaussian action corresponding to the 
left-regular representation, and the circle group $\T$ is contained in the class $\UU_{\rm fin}$.  When combined with Corollary 2.4 in \cite{petersonthom} which states that for a nonamenable group vanishing of the first $\ell^2$-Betti number is equivalent to $H^1(\G, \lambda) = \{ 0 \}$ we obtain the following corollary.

\begin{cor}
Let $\G$ be a countable discrete group.  If $\beta_1^{(2)}(\G) \not= 0$ then the Bernoulli shift action is not
$\UU_{\rm fin}$-cocycle superrigid.
\end{cor}

The second theme explored is the deformation/derivation duality developed by the first author in \cite{petersonl2}.  The flexibility inherent at the infinitesimal level allows us to offer a unified treatment of Popa's theorem in the case of generalized Bernoulli actions and expand the class of groups whose Bernoulli actions are known to be $\UU_{\mathrm {fin}}$-cocycle superrigid. As a partial converse to the above results, we have that an \textit{a priori} stronger property than having $\beta_1^{(2)}(\G) = 0$, $L^2$-rigidity (see Definition \ref{defn:l2rigid}), is sufficient to guarantee $\UU_{\mathrm {fin}}$-cocycle superrigidity of the Bernoulli shift.

\begin{thm} 
Let $\G$ be a countable discrete group.  If $L\G$ is $L^2$-rigid then the Bernoulli shift action of $\G$ is $\UU_{\rm fin}$-cocycle superrigid.
\end{thm}

Examples of groups for which this holds are groups which contain an infinite normal subgroup which has relative property (T) or is the direct product of an infinite group
and a nonamenable group, recovering Popa's Cocycle Superrigidity Theorem for Bernoulli actions of these groups. 

We also obtain new groups for which Popa's theorem holds.  For example, we show that the theorem holds for any
generalized wreath product $A_0 \wr_X \G_0$, where $A_0$ is a non-trivial abelian group and $\G_0$ does not have the Haagerup property.
Also, if $L\Lambda$ is nonamenable and has property $(\G)$ of Murray and von Neumann \cite{mvn4} then the theorem also holds for $\Lambda$.

We remark that it is still an open question whether vanishing of the first $\ell^2$-Betti number characterizes groups whose Bernoulli actions are $\UU_{\rm fin}$-cocycle superrigid.  For instance, it is still unknown for the group $\Z \wr \F_2$, which contains an infinite normal abelian subgroup and hence has vanishing first $\ell^2$-Betti number by \cite{cheegergromov}.

The authors would like to thank Sorin Popa for useful discussions regarding this work.

\section{Preliminaries}  We begin by reviewing of the basic notions of Gaussian actions, cohomology of representations and actions, and closable derivations. Though our treatment of the last two topics is standard, our approach to Gaussian actions is somewhat non-standard, where we take a more operator-algebraic approach by viewing the algebra of bounded functions on the probability space as a von Neumann algebra acting on a symmetric Fock space.  In the noncommutative setting of free probability this is the same as Voiculescu's approach in \cite{voiculescusymmetries}.
But first, let us recall a few basic definitions and concepts which constitute the basic language in which this paper is written.  Throughout, all Hilbert spaces are assumed to be separable.

\begin{defn}  Let $\pi: \G \rightarrow \UU(\HH)$ be a unitary representation, and denote by $\pi^{op}$ the associated contragradient representation on the contragradient Hilbert space $\HH^{op}$ of $\HH$.  We say that $\pi$:
\begin{enumerate}
	\item is \textbf{ergodic} if $\pi$ has no non-zero invariant vectors;
	
	\item is \textbf{weakly mixing} if $\pi \otimes \pi^{op}$ is ergodic (equivalently, $\pi \otimes \rho^{op}$ is ergodic for any unitary $\G$-representation $\rho$);
	
	\item is \textbf{mixing} if $\ip{\pi_\g(\xi)}{\eta} \rightarrow 0$ as $\g \rightarrow \infty$, for all $\xi, \eta \in \HH$;
	
	\item has \textbf{spectral gap} if there exists $K \subset G$, finite, and $C >0$ such that $\|\xi - P(\xi)\| \leq C \sum_{k \in K}\|\pi_k(\xi) -\xi\|$, for all $\xi \in \HH$, where $P$ is the projection onto the invariant vectors;

	\item has \textbf{stable spectral gap} if $\pi \otimes \pi^{op}$ has spectral gap (equivalently, $\pi \otimes \rho^{op}$ has spectral gap for any unitary $\G$-representation $\rho$);
	
	\item is \textbf{amenable} if $\pi$ is either not weakly mixing or does not have stable spectral gap.
	
\end{enumerate}
\label{def:ergodic}
\end{defn}

Note that for an orthogonal representation $\pi$ of $\G$ into a real Hilbert space $\KK$, the associated unitary representation into $\KK \otimes \C$ is canonically isomorphic to its contragradient. Hence, in this situation we may replace in the above definition ``$\pi \otimes \pi^{op}$'' and ``$\pi \otimes \rho^{op}$'' with ``$\pi \otimes \pi$'' and ``$\pi \otimes \rho$'', respectively.

Let $\G \car^\sigma (X,\mu)$ be an action of the countable discrete group $\G$ by $\mu$-preserving automorphisms of a standard probability space $(X,\mu)$.  This yields a unitary representation $\pi^\sigma : \G \rightarrow \UU(L_0^2(X,\mu))$ called the \textbf{Koopman representation} associated to $\sigma$. (Here $L_0^2(X,\mu)$ denotes the orthogonal complement in $L^2(X,\mu)$ to the subspace of the constant functions on $X$.)  Note that the Koopman representation is the unitary representation associated to the orthogonal representation of $\G$ acting on the real-valued $L^2$-functions.  We say that the action $\sigma$ is ergodic (or weakly mixing, mixing, etc.) if the Koopman representation $\pi^\sigma$ is in the sense of the above definition.  An action $\G \car^\sigma (X,\mu)$ is \textbf{(essentially) free} if, for all $\g \in \G$,
$\g \not= e$, $\mu \{ x \in X : \sigma_\g(x) = x \} = 0$.

Given unitary representations $\pi : \G \rightarrow \UU(\HH)$ and $\rho : \G \rightarrow \UU(\KK)$, we say that $\pi$ is \textbf{contained} in $\rho$ if there is a linear isometry $V: \HH \rightarrow \KK$ such that $\pi_\g = V^* \rho_\g V$, for all $\g \in \G$.  We say that $\pi$ is \textbf{weakly contained} $\rho$ if for any $\xi, \eta \in \HH$, there are $\xi_n', \eta_n' \in \KK$ such that $\ip{\pi_\g(\xi)}{\eta} = \Sigma_{n \in \N}\ip{\rho_\g(\xi_n')}{\eta_n'}$, for all $\g \in \G$.  Note that amenability of a representation $\pi$ is equivalent to $\pi \otimes \pi^{op}$ weakly containing the trivial representation, which is equivalent with Bekka's definition by Theorem 5.1 in \cite{bekka}.\\

The ``representation theory" of a $\mathrm {II}_1$ factor is captured in the structure of its bimodules, also called correspondences (cf. \cite{popacorrespondence}).  The theory of correspondences was first developed by A. Connes \cite{connescorrespondence}.

\begin{defn} Let $M$ be a $\mathrm {II}_1$ factor.  An \textbf{$M$-$M$ Hilbert bimodule} is a Hilbert space $\HH$ equipped with a representation $\pi: M \otimes_{\mathrm{alg}} M^{op} \rightarrow B(\HH)$ which is normal when restricted to $M$ and $M^{op}$.  We write $\pi(x \otimes y^{op})\xi$ as $x \xi y$.
\end{defn}

An $M$-$M$ Hilbert bimodule $\HH$ is \textbf{contained} in an $M$-$M$ Hilbert bimodule $\KK$ if there is a linear isometry $V: \HH \rightarrow \KK$ such that $V(x \xi y) = x V(\xi) y$, for all $\xi \in \HH$, $x,y \in M$; $\HH$ is \textbf{weakly contained} in $\KK$ if for any $\xi, \eta \in \HH$, $F \subset M$ finite, there exist $\xi_n', \eta_n' \in \KK$ such that $\ip{x \xi y}{\eta} = \Sigma_{n \in \N} \ip{x \xi_n' y}{\eta_n'}$, for all $x,y \in F$.  The \textbf{trivial bimodule} is the space $L^2(M,\tau)$ with the obvious bimodule structure induced by left and right multiplication; the \textbf{coarse bimodule} is the space $L^2(M,\tau)\otimes L^2(M,\tau)$ with bimodule structure induced by left multiplication on the first factor and right multiplication on the second.  The trivial and coarse bimodules play analogous roles in the theory of $M$-$M$ Hilbert bimodules to the roles played, respectively, by the trivial and left-regular representations in the theory of unitary representations of locally-compact groups.  Note that an $M$-$M$ correspondence $\HH$ contains the trivial correspondence if and only if $\HH$ has non-zero $M$-central vectors (a vector $\xi$ is $M$-central if $x \xi = \xi x$, for all $x \in M$). 

Given $\xi, \eta \in \HH$, note the maps $M \ni x \mapsto \ip{x \xi}{\eta}, \ip{\xi x}{\eta}$ are normal linear functionals on $M$.  A vector $\xi \in \HH$ is called \textbf{left (respectively, right) bounded} if there exists $C >0$ such that for every $x \in M$, $\|x \xi\| \leq C \|x\|_2$, (resp., $\|\xi x\| \leq C \|x\|_2$).  The set of vectors which are both left and right-bounded forms a dense subspace of $\HH$.  By \cite{popacorrespondence}, to $\xi$, a left-bounded vector, we can associate a completely-positive map $\phi_\xi: M \rightarrow M$ such that for all $x,y \in M$, $\|x \xi y\| = \tau(x^*x\phi_\xi(yy^*))^{1/2}$.  If $\xi$ is also right-bounded then this map is seen to naturally extend to a bounded operator $\hat \phi_\xi : L^2(M,\tau) \rightarrow L^2(M,\tau)$.

Given two $M$-$M$ Hilbert bimodules $\HH$ and $\KK$, there is a well-defined tensor product $\HH \otimes_M \KK$ in the category of $M$-$M$ Hilbert bimodules: see \cite{popacorrespondence} for details.

\begin{defn}[Compare with Definition \ref{def:ergodic}.]  Let $\HH$ be an $M$-$M$ Hilbert bimodule.  $\HH$ is said to:
\begin{enumerate}
	\item be \textbf{weakly mixing} if $\HH \otimes_M \HH^{op}$ does not contain the trivial $M$-$M$ Hilbert bimodule;

	\item be \textbf{compact (or mixing)} if for every sequence $u_i \in \UU(M)$ such that $u_i \rightarrow 0$, weakly, we have that $$\lim_{i \rightarrow \infty}\sup_{\|x\| \leq 1}\ip{u_i \xi x}{\eta} = \lim_{i \rightarrow \infty}\sup_{\|x\| \leq 1}\ip{x \xi u_i}{\eta} = 0,$$ for all $\xi, \eta \in \HH$ (equivalently, $\hat \phi_\xi$ is a compact operator on $L^2(M,\tau)$ for every left-bounded vector $\xi \in \HH$);
	
	\item have \textbf{spectral gap} if there exist $x_1, \dots, x_n \in M$ such that $\|\xi - P(\xi)\| \leq \sum_{i=1}^n\|x_i \xi - \xi x_i\|$, for all $\xi \in \HH$, where $P$ is the projection onto the central vectors;
	
	\item have \textbf{stable spectral gap} if  $\HH \otimes_M \HH^{op}$ has spectral gap;
	
	\item be \textbf{amenable} if it is either not weakly mixing or does not have stable spectral gap.
\end{enumerate}
\end{defn}

If $\HH$ is a compact $M$-correspondence and $\KK$ an arbitary $M$-correspondence, then $\HH \otimes_M \KK$ (and also $\KK \otimes_M \HH$) is compact, since $\hat \phi_{\xi \otimes_M \eta} = \hat \phi_\eta \circ \hat \phi_\xi$ if $\xi$ and $\eta$ are both left and right-bounded.\\

Let $\HH$ and $\KK$ be $M$-$M$ correspondences, and denote by $\HH_0$ and $\KK_0$ the set of right-bounded vectors in $\HH$ and $\KK$, respectively.  For $\xi, \eta \in \HH_0$, denote by $(\xi|\eta)$ the element of $M$ such that $\ip{\xi x}{\eta y} = \tau(y^*(\xi|\eta)x)$, for all $x, y \in X$ (by normality of the map $xy^* \mapsto \ip{\xi xy^*}{\eta}$, there exists such a $(\xi|\eta) \in L^1(M,\tau)$; right-boundedness of $\xi$ and $\eta$ implies $(\xi|\eta) \in M$).  It is clear that $( \cdot | \cdot )$ is a bilinear map $\HH_0 \times \HH_0 \rightarrow M$ such that $(\xi | \xi) \geq 0$ and $(\xi | \xi) = 0$ if and only if $\xi = 0$, for all $\xi \in \HH_0$.  For $\xi \in \HH_0$ and $\eta \in \KK_0$, define the linear map $T_{\xi,\eta} : \HH_0 \rightarrow \KK_0^{op}$ by $T_{\xi,\eta}(\cdot) = (\cdot|\xi)\eta^{op}$.  It is easy to check that $T_{\xi,\eta}$ is a bounded with $\|T_{\xi,\eta}\| = \|\xi\|^{-1/2}\|(\xi|\xi)\eta\|$; hence, $T_{\xi,\eta}$ extends to a bounded operator $\HH \rightarrow \KK^{op}$.  Let $\LL_M^2(\HH,\KK)$ be the subspace of $B(\HH,\KK^{op})$ which is the closed span of all operators of the form $T_{\xi,\eta}$ under the Hilbert norm $\|T_{\xi,\eta}\|_{\LL_M^2} = \tau((\xi|\xi)(\eta|\eta))^{1/2}$.  Moreover, $\LL_M^2(\HH,\KK)$ is equipped with a natural $M$-$M$ Hilbert bimodule structure given by $(x \otimes y^{op})(T_{\xi,\eta}) = T_{x \xi, y \eta}$ identifying it with $\HH \otimes_M \KK^{op}$.  Note that if $T \in \LL_M^2(\HH,\KK)$, then $(T^*T)^{1/2} \in \LL_M^2(\HH,\HH)$.

\begin{prop} An $M$-$M$ correspondence $\HH$ is weakly mixing if and only if for any $M$-$M$ correspondence $\KK$, $\HH \otimes_M \KK^{op}$ does not contain the trivial correspondence.
\end{prop}

\begin{proof}  The reverse implication is trivial.  Conversely, suppose there exists $\KK$ such that $\HH \otimes \KK^{op}$ contains an $M$-central vector.  Identifying $\HH \otimes_M \KK^{op}$ with $\LL_M^2(\HH,\KK)$, let $T \in \LL_M^2(\HH,\KK)$ be an $M$-central vector. Then $(T^*T)^{1/2} \in \LL_M^2(\HH,\HH)$ is an $M$-central vector; hence, $\HH$ is not weakly mixing.
\end{proof}


\subsection{Gaussian actions} \label{sec:gaussian}

Let $\pi:\G \rightarrow \OO(\HH)$ be an orthogonal representation of a countable discrete group $\G$. The aim of this section is to describe the construction of a measure-preserving action of $\G$ on a non-atomic standard probability space $(X,\mu)$ such that $\HH$ is realized as a subspace of $L_\R^2(X,\mu)$ and $\pi$ is contained in the Koopman representation $\G \car L_0^2(X,\mu)$.  The action $\G \car (X,\mu)$ is referred to as the \textbf{Gaussian action} associated to $\pi$.  We give an operator-algebraic alternative construction of the Gaussian action similar to Voiculescu's construction of free semi-circular random variables.\\

Given a real Hilbert space $\HH$, the \textbf{n-symmetric tensor} $\HH^{\odot n}$ is the subspace of $\HH^{\otimes n}$ fixed by the action of the symmetric group $S_n$ by permuting the indices.  For $\xi_1, \dots, \xi_n \in \HH$, we define their symmetric tensor product $\xi_1 \odot \dots \odot \xi_n \in \HH^{\odot n}$ to be $\frac{1}{n!} \sum_{\sigma \in S_n} \xi_{\sigma(1)} \otimes \dots \otimes \xi_{\sigma(n)}$.
Denote $$\SSS(\HH) = \C \Omega \oplus \bigoplus_{n=1}^\infty (\HH \otimes \C)^{\odot n},$$ with renormalized inner product such that $\|\xi\|_{\SSS(\HH)}^2 = n!\|\xi\|^2$, for $\xi \in \HH^{\odot n}$.\\

For $\xi \in \HH$ let $x_\xi$ be the \textbf{symmetric creation operator}, $$x_\xi(\Omega) = \xi, \ x_\xi(\eta_1 \odot \dots \odot \eta_k) = \xi \odot \eta_1 \odot \dots \odot \eta_k,$$ and its adjoint, $\frac{\partial}{\partial \xi} = (x_\xi)^*$
$$\frac{\partial}{\partial \xi}(\Omega) = 0, \ \frac{\partial}{\partial \xi}(\eta_1 \odot \dots \odot \eta_k) = \sum_{i=1}^k \ip{\xi}{\eta_i}\eta_1 \odot \dots \odot \widehat \eta_i \odot \dots \odot \eta_k.$$
Let $$s(\xi) = \frac{1}{2} (x_\xi + \frac{\partial}{\partial \xi}),$$ and note that it is an unbounded, self-adjoint operator on $\SSS(\HH)$.\\

The \textbf{moment generating function} $M(t)$ for the Gaussian distribution is defined to be $$M(t) = \frac{1}{\sqrt{2 \pi}}\int_{-\infty}^\infty \exp(tx) \exp(-x^2/2)dx = \exp(t^2/2).$$ It is easy to check that if $\|\xi\| =1$ then $$\langle s(\xi)^n \, \Omega, \Omega \rangle = M^{(n)}(0) = \frac{(2k)!}{2^k k!},$$ if $n = 2k$ and $0$ if $n$ is odd. Hence, $s(\xi)$ may be regarded as a Gaussian random variable.  Note that if $\xi, \eta \in \HH$ then $s(\xi)$ and $s(\eta)$ commute, moreover, if $\xi \perp \eta$, then $\ip{s(\xi)^m s(\eta)^n \Omega}{\Omega} = 0$, for all $m,n \in \N$; thus, $s(\xi)$ and $s(\eta)$ are independent random variables.\\

From now on we will use the convention $\xi_1 \xi_2 \dots \xi_k$ to denote the symmetric tensor $\xi_1 \odot \xi_2 \odot \dots \odot \xi_k$.  Let $\Xi$ be a basis for $\HH$ and $$\mathcal S(\Xi) = \{\Omega\} \cup \{s(\xi_1)s(\xi_2) \dots s(\xi_k)\Omega : \xi_1, \xi_2, \dots, \xi_k \in \Xi \}.$$

\begin{lem} The set $\mathcal S(\Xi)$ is a (non-orthonormal) basis of $\SSS(\HH)$.
\label{lem:sbasis}
\end{lem}

\begin{proof}  We will show that $\xi_1 \dots \xi_k \in \mathrm{span}(\mathcal S(\Xi))$, for all $\xi_1, \dots, \xi_k \in \HH$.  We have $\Omega \in \mathrm{span}(\mathcal S(\Xi))$.  Also, since $s(\xi)\Omega = \xi$, $\HH \subset \mathrm{span}(\mathcal S(\Xi))$.  Now as $s(\xi_1) \dots s(\xi_k)\Omega = P(\xi_1, \dots, \xi_k)$ is a polynomial in $\xi_1, \dots, \xi_k$ of degree $k$ with top term $\xi_1 \dots \xi_k$, the result follows by induction on $k$.
\end{proof}

Let $u(\xi_1, \dots, \xi_k) = \exp(\pi i s(\xi_1) \dots s(\xi_k))$ and $u(\xi_1, \dots, \xi_k)^t = \exp(\pi i t s(\xi_1) \dots s(\xi_k))$. Denote by $A$ the von Neumann algebra generated by all such $u(\xi_1, \dots, \xi_k)$, which is the same as the von Neumann algebra generated by the spectral projections of the unbounded operators $s(\xi_1) \dots s(\xi_k)$.

\begin{thm} We have that $L^2(A,\tau) \cong \SSS(\HH)$, and $A$ is a maximal abelian $*$-subalgebra of $B(\SSS(\HH))$ with faithful trace $\tau = \ip{\cdot \, \Omega}{\Omega}$.  In particular, $A$ is a diffuse abelian von Neumann algebra.
\end{thm}

\begin{proof} By Lemma \ref{lem:sbasis}, $A \mapsto A \Omega$ is an embedding of $A$ into $\SSS(\HH)$.  By Stone's Theorem $$\lim_{t\rightarrow 0} \frac{u(\xi_1, \dots, \xi_k)^t - 1}{\pi i t} \Omega = s(\xi_1) \dots s(\xi_k) \Omega;$$ hence, $A \Omega$ is dense in $\SSS(\HH)$.  This implies that $A$ is maximal abelian in $B(\SSS(\HH))$.
\end{proof}

There is a natural strongly-continuous embedding $\OO(\HH) \hookrightarrow \UU(\SSS(\HH))$ given by $$T \mapsto T^\SSS = 1 \oplus \bigoplus_{n=1}^\infty T^{\odot n}.$$  It follows that there is an embedding $\OO(\HH) \rightarrow \Aut(A,\tau)$, $T \mapsto \sigma_T$, which can be identified on the unitaries $u(\xi_1, \dots, \xi_k)$ by $$\sigma_T(u(\xi_1, \dots, \xi_k)) = \Ad(T^\SSS)(u(\xi_1, \dots, \xi_k)) = u(T(\xi_1), \dots, T(\xi_k)).$$  Thus for an orthogonal representation $\pi: \G \rightarrow \OO(\HH)$, there is a natural action $\sigma^\pi : \G \rightarrow \Aut(A,\tau)$ given by $\sigma_\g^\pi(u(\xi_1, \dots, \xi_k)) = u(\pi_\g(\xi_1), \dots, \pi_\g(\xi_k)) = \Ad(\pi_\g^\SSS)(u(\xi_1, \dots, \xi_k))$.  The action $\sigma^\pi$ is the \textbf{Gaussian action} associated to $\pi$.\\

We have that ergodic properties which remain stable with respect to tensor products tranfer from $\pi$ to $\sigma^\pi$.

\begin{prop}  In particular, for a subgroup $H \leq \G$, $\sigma^\pi|_H$ possesses any of the following properties if and only if $\pi|_H$ does:
\begin{enumerate}
	\item weak mixing;
	
	\item mixing;
	
	\item stable spectral gap;
	
	\item being contained in a direct sum of copies of the left-regular representation;
	
	\item being weakly contained in the left-regular representation.
	
\end{enumerate}
\label{prop:ergodic}
\end{prop}

For Gaussian actions, stable properties are equivalent to their ``non-stable" counterparts.  The following proposition serves as a prototype of such a result, showing that ergodicity implies stable ergodicity, \textit{i.e.}, weak mixing.

\begin{thm} $\G \car^{\sigma^\pi}(A,\tau)$ is ergodic if and only if $\pi$ is weakly mixing.
\end{thm}

\begin{proof}
The reverse implication follows from Proposition \ref{prop:ergodic}.  Conversely, suppose there exists $\xi \in \HH^{\otimes 2}$ such that for all $\g \in \G$, $\pi_\g^2(\xi) = \xi$.  Viewing $\xi$ as a Hilbert-Schmidt operator on $\HH$, let $|\xi| = (\xi\xi^*)^{1/2}$.  Since the map $\xi \otimes \eta \mapsto \eta \otimes \xi$ is the same as taking the adjoint of the corresponding Hilbert-Schmidt operator, we have that $|\xi| \in \HH^{\odot 2}$ and $\pi_\g(|\xi|) = |\xi|$.  By functional calculus, there exists $\lambda > 0$, such that $\eta = E_\lambda(|\xi|) \neq 0$ is a finite rank operator.  Hence, $\eta = \eta_{11} \odot \eta_{12} + \dots + \eta_{n1} \odot \eta_{n2} \in \HH^{\odot 2}$ with $\eta_{i1} \odot \eta_{i2} \perp \eta_{j1} \odot \eta_{j2}$ for $i \neq j$.  But then $u = \prod_{i=1}^n u(\eta_{i1},\eta_{i2}) \in A$, a non-trivial unitary and $\sigma_\g^\pi(u) = u$.  Hence, $\sigma^\pi$ is not ergodic.
\end{proof}


\subsection{Cohomology of a representation \& action}

Let $\KK$ be a real Hilbert space and $\pi:\G\rightarrow\OO(\KK)$ an orthogonal representation of a countable discrete group $\G$.

\begin{defn} A \textbf{cocycle} is a map $b:\G\rightarrow\KK$ satisfying the cocycle identity $b(\g_1\g_2)=\pi_{\g_1} b(\g_2) + b(\g_1)$, for all $\g_1,\g_2\in\G$.  A cocycle is a \textbf{coboundary} is there exists $\eta\in\KK$ such that $b(\g) = \pi_\g\eta -\eta$, for all $\g\in\G$.
\end{defn}

It is a well-known fact (cf. \cite{bdhv}) that a cocycle $b$ is a coboundary if and only if $\sup_{\g\in\G}\| b(\g) \|<\infty$.  Let $Z^1(\G,\pi)$ and $B^1(\G,\pi)$ denote, respectively, the vector space of all cocycles and the subspace of coboundaries.  The \textbf{first cohomology space} $H^1(\G,\pi)$ of the representation $\pi$ is then defined to be $Z^1(\G,\pi)/B^1(\G,\pi)$.\\

Let $\G\car^\sigma(X,\mu)$ be an ergodic, measure-preserving action on a standard probability space $(X,\mu)$, and let $\A$ be a Polish topological group.

\begin{defn}  A \textbf{cocycle} is a measurable map $c:\G\times X \rightarrow \A$ satisfying the cocycle identity $c(\g_1\g_2,x) = c(\g_1,\sigma_{\g_2}(x))c(\g_2,x)$, for all $\g_1,\g_2\in\G$, $\mathrm{a.e.}\ x\in X$. A pair of cocycles $c_1, c_2$ are \textbf{cohomologous} (written $c_1 \sim c_2$) if there exists a measurable map $\xi: X \rightarrow \A$ such that $\xi(\sigma_\g(x))c_1(\g,x)\xi(x)^{-1} = c_2(\g,x)$ for all $\g \in \G$, $\mathrm{a.e.}\ x \in X$.  A cocycle is a \textbf{coboundary} if it is cohomologous to the cocycle which is identically 1.
\end{defn}

Let $Z^1(\G,\sigma,\A)$ and $B^1(\G,\sigma,\A)$ denote, respectively, the space of all cocycles and the subspace of coboundaries.  The \textbf{first cohomology space} $H^1(\G,\sigma,\A)$ of the action $\sigma$ is defined to be $Z^1(\G,\sigma,\A)/ \sim$.  Note that if $\A$ is abelian, $Z^1(\G,\sigma,\A)$ is endowed with a natural abelian group structure and $H^1(\G,\sigma,\A) = Z^1(\G,\sigma,\A)/B^1(\G,\sigma,\A)$.  To any homomorphism $\rho: \G \rightarrow \A$ we can associate a cocycle $\tilde{\rho}$ by $\tilde{\rho}(\g,x) = \rho(\g)$.  Using terminology developed by Popa (cf. \cite{popasuperrigid}), a cocycle $c$ is said to \textbf{untwist} if there exists a homomorphism $\rho:\G \rightarrow \A$ such that $c$ is cohomologous to $\tilde{\rho}$.  To any cocycle $c \in Z^1(\G,\sigma,\A)$, we can associated two cocycles $c_\ell, c_r \in Z^1(\G,\sigma \times \sigma,\A)$ given by $c_\ell(\g,(x,y)) = c(\g,x)$ and $c_r(\g,(x,y)) = c(\g,y)$.  It is easy to check that $c$ untwists only if $c_\ell$ is cohomologous to $c_r$; if $\sigma$ is weakly mixing, Theorem 3.1 in \cite{popasuperrigid} establishes the converse.


\subsection{Closable derivations}

We review here briefly some general properties of closable derivations on a finite von Neumann algebra and set up some notation to be used in the sequel.  
For a more detailed discussion see \cite{davieslindsay}, \cite{petersonT}, \cite{petersonl2}, or \cite{ozawapopaII}.

\begin{defn} Let $(N,\tau)$ be a finite von Neumann algebra and $\HH$ be an $N$-$N$ correspondence. A \textbf{derivation} $\dd$ is an unbounded operator $\dd:L^2(N,\tau)\rightarrow\HH$ such that $D(\dd)$ is a $\|\cdot\|_2$-dense $*$-subalgebra of $N$, and $\dd(xy) = x\dd(y)+\dd(x)y$, for each $x,y\in D(\dd)$.  A derivation is \textbf{closable} if it is closable as an operator and \textbf{real} if $\HH$ has an antilinear involution $\JJ$ such that $\JJ (x \xi y) = y^* \JJ(\xi) x^*$, and $\JJ(\delta(z)) = \delta(z^*)$, for each $x,y \in N$, $\xi \in \HH$, $z \in D(\dd)$.
\end{defn}

If $\dd$ is a closable derivation then by \cite{davieslindsay} $D(\overline\dd) \cap N$ is again a $*$-subalgebra and $\dd_{|D(\overline\dd \cap N)}$ is again
a derivation.  We will thus use the slight abuse of notation by saying that $\overline\dd$ is a closed derivation.

To every closed real derivation $\delta:N \rightarrow \HH$, we can associate a semigroup deformation $\Phi^t = \exp(-t\delta^*\overline\delta)$, $t > 0$, 
and a resolvent deformation $\zeta_\alpha = (\alpha/(\alpha + \delta^*\overline{\delta}))^{1/2}$, $\alpha > 0$.  Both of these deformations are of unital, 
symmetric, completely-positive maps; moreover, the derivation $\delta$ can be recovered from these deformations \cite{sauvageot1, sauvageot2}.

We also have that the deformation $\Phi^t$ converges uniformly on $(N)_1$ as $t \rightarrow 0$ if and only if 
the deformation $\zeta_\alpha$ converges uniformly on $(N)_1$ as $\alpha \rightarrow \infty$.

\begin{defn}[Definition 4.1 in ~\cite{petersonl2}]\label{defn:l2rigid}
Let $(N, \tau)$ be a finite von Neumann algebra. $N$ is \textbf{$L^2$-rigid} if 
given any inclusion $(N, \tau) \subset (M, \tilde \tau)$, and any closable real derivation $\delta: M \rightarrow \HH$
such that $\HH$ when viewed as an $N$-$N$ correspondence embeds in $( L^2N \overline \otimes L^2N )^{\oplus \infty}$,
we then have that the associated deformation $\zeta_\alpha$ converges uniformly to the identity in $\| \cdot \|_2$ on the unit ball of $N$.
\end{defn}

We point out here that our definition above is formally stronger than the one given in \cite{petersonl2}.  Specifically, there it was assumed that $\HH$ embedded into the coarse bimodule as an $M$-$M$ bimodule rather than an $N$-$N$ bimodule.  However, this extra condition was not used in \cite{petersonl2}, and since the above definition has better stability properties (see Theorem \ref{thm:l2rigidityoe}) we have chosen to use the same terminology.

Examples of nonamenable groups which do not give rise to $L^2$-rigid group von Neumann algebras are groups
such that the first $\ell^2$-Betti number is positive.  These are, in fact, the only known examples, 
and $L^2$-rigidity should be viewed as a von Neumann analog of vanishing first $\ell^2$-Betti number.

Showing that a group von Neumann algebra is $L^2$-rigid can be quite difficult in general since one has to consider derivations which may
not be defined on the group algebra.  Nonetheless, there are certain situations where this can be verified. 

\begin{thm}[Corollary 4.6 in ~\cite{petersonl2}]
Let $\G$ be a nonamenable countable discrete group.  If $L\G$ is weakly rigid, non-prime, or has property $(\G)$ of Murray and von Neumann, then
$L\G$ is $L^2$-rigid.
\end{thm}

We give another class of examples below (see also \cite{ozawapopaI}, \cite{ozawapopaII}, or \cite{petersonwe}).  The gap between group von Neumann
algebras which are known to be $L^2$-rigid  and groups with vanishing first $\ell^2$-Betti number is, however, quite large.  
For example, as we mentioned in the introduction, the wreath product $\Z \wr \F_2$ is a group which has vanishing first $\ell^2$-Betti number but for which it is not known whether the group von Neumann algebra is $L^2$-rigid.

\section{Deformations}  

In this section and Section \ref{sec:derivations} we will discuss the interplay between one-parameter groups of automorphisms or, more generally, semigroups of completely-positive maps of finite factors (deformations) and their infinitesimal generators (derivations).  The motivation for studying deformations at the infinitesimal level is that it allows for the creation of other related deformations of the algebra. And while Popa's deformation/rigidity machinery requires uniform convergence of the original deformation on some target subalgebra, it is often more feasible to demonstrate uniform convergence of a related deformation, then transfer those estimates back the the original.  

We begin by recalling Popa's notion of an s-malleable deformation, and give some examples of such deformations that have appeared in the literature.

\begin{defn}[Definition 4.3 in \cite{popasuperrigid}] Let $(M,\tau)$ be a finite von Neumann algebra such that $(M,\tau) \subset (\tilde{M},\tilde{\tau})$.  A pair $(\alpha,\beta)$, consisting of a point-wise strongly continuous one-parameter family $\alpha : \R \rightarrow \Aut(\tilde M, \tilde \tau)$ and an involution $\beta\in\Aut(\tilde{M},\tilde{\tau})$ is called a \textbf{s-malleable deformation} of $M$ if:

\begin{enumerate}
  \item $M\subset\tilde{M}^\beta$;
  
	\item $\alpha_t\circ\beta = \beta\circ\alpha_{-t}$; and 
	
	\item $\alpha_1(M)\perp M$.
\end{enumerate}
\end{defn}

\subsection{Popa's deformation}\label{popasdeformation}  
The following deformation was used by Popa in \cite{popasuperrigid} to obtain cocycle superrigidity for generalized Bernoulli actions of property (T) groups.

Let $(A,\tau)$ be a finite diffuse abelian von Neumann algebra and $u,v\in A\otimes A$ be generating Haar unitaries for $A\otimes 1, 1\otimes A\subset A\otimes A$, respectively.  Set $w=u^*v$. Choose $h\in A\otimes A$ self-adjoint such that $\exp(\pi i h) = w$, and let $w^t = \exp(\pi i t h)$. Since $\{w\}''\perp A\otimes 1, 1\otimes A$, we have that for any $t$, $w^tu$ and $w^tv$ are again Haar unitaries.  Moreover, since $w\in\{w^tu,w^tv\}''$, $\{w^tu, w^tv\}$ is a pair of generating Haar unitaries in $A\otimes A$.  Hence there is a well-defined one-parameter family $\alpha:\R\rightarrow\Aut(A\otimes A,\tau\otimes\tau)$ given by 
$$
\alpha_t(u) = w^tu,\ \alpha_t(v) = w^tv.
$$  
The family $\alpha$, together with the automorphism $\beta$ given by 
$$
\beta(u) = u,\ \beta(v)=u^2v^*,
$$
is seen to be an s-malleable deformation of $A\otimes 1\subset A\otimes A$.

\begin{defn} Let $(P,\tau)$ be a finite von Neumann algebra and $\sigma:\G\rightarrow\Aut(P,\tau)$ a $\G$-action.  $\G\car^\sigma(P,\tau)$ is an \textbf{s-malleable action} if there exists an s-malleable deformation $(\alpha,\beta)$ of $(P,\tau)$ such that $\beta, \alpha_t$ commute with $\sigma_\g\otimes\sigma_\g$ for all $t\in\R,\ \g\in\G$.
\end{defn}

For any countable discrete group there is a canonical example of an s-malleable action, the \textbf{Bernoulli shift}.  Let $(A,\tau) = (L^\infty(\T,\lambda),\int\cdot d\lambda)$, $(X,\mu) = \prod_{g\in\G}(\T,\lambda)$, and $(B,\tau') = \bigotimes_{\g\in\G}(A,\tau)$.  The Bernoulli shift is the natural action $\G\car^\sigma(X,\mu)$ defined by shifting indices: $\sigma_{\g_0}((x_\g)_\g) = (x_\g)_{\g_0\g} = (x_{\g_0^{-1}\g})_\g$.  Defining 
$$
\tilde \alpha_t((\tilde x_\g)_\g) = (\alpha_t( \tilde x_\g))_\g
$$ 
and 
$$
\tilde \beta((\tilde x_\g)_\g) = (\beta(\tilde x_\g))_\g,
$$ for $(\tilde x_\g)_\g \in \tilde B = \bigotimes_{\g\in\G}(A \otimes A) \cong B \otimes B$, we see that $(\tilde \alpha, \tilde \beta)$ is an s-malleable deformation of $B$ which commutes with the Bernoulli $\G$-action.


\subsection{Ioana's deformation}\label{sec:chifanioana}  The deformation described below was first used by Ioana \cite{ioanawreath} in the case when the base space is nonamenable, and later used by Chifan and Ioana \cite{chifanioana} in part to obtain solidity of $L^\infty(X,\mu)\rtimes_\sigma \G$, whenever $L\G$ solid and $\G \car^\sigma (X,\mu)$, the Bernoulli shift.  Their deformation is inspired by the free product deformation used in \cite{ipp}.  A similar deformation has also been previously used by Voiculescu in \cite{voiculescufreeproduct}.

Given a finite von Neumann algebra $(B,\tau)$, let $\tilde{B} = B*L(\Z)$.  If $u\in\UU(L(\Z))$ is a generating Haar unitary, choose an $h\in L(\Z)$ such that $\exp(\pi i h) = u$, and let $u^t = \exp(\pi i t h)$.  Define the deformation $\alpha:\R\rightarrow\Aut(\tilde{B},\tilde \tau)$ by 
$$
\alpha_t = \Ad(u^t).
$$  
Let $\beta\in\Aut(\tilde{B},\tilde{\tau})$ be defined by 
$$
\beta_{|B} = \id \ \ {\rm and} \ \  \beta(u) = u^*.
$$  
It is easy to check that $(\alpha, \beta)$ is a s-malleable deformation of $B$.  

If a countable discrete group $\G$ acts on a countable set $S$ then we may consider the \textbf{generalized Bernoulli shift action} of $\G$ on $\otimes_{s \in S} B$ given
by $\sigma_\g( \otimes_{s \in S} b_s ) = \otimes_{s \in S} b_{\g^{-1}s}$.  We then have that $\otimes_{s \in S} B \subset \otimes_{s \in S} \tilde{B}$ and 
$( \otimes_{s \in S} \alpha, \otimes_{s \in S} \beta)$ gives a s-malleable deformation of $\otimes_{s \in S} B$.


\subsection{Malleable deformations of Gaussian actions}\label{sec:infinitesimalgaussian}
We will now construct the canonical s-malleable deformation of a Gaussian action which is given in Section 4.3 of \cite{furmancsr}, and give an explicit description of its associated derivation.  To begin, let $\pi: \G \rightarrow \OO(\HH)$ be an orthogonal representation, $\tilde \HH = \HH \oplus \HH$, and $\tilde \pi = \pi \oplus \pi$. If $\sigma^\pi: \G \rightarrow \Aut(A,\tau)$ is the Gaussian action associated with $\pi$, then the Gaussian action associated to $\tilde \pi$ is naturally identified with the action $\sigma^\pi \otimes \sigma^\pi$ on $A \otimes A$.  Let $\tilde \sigma^\pi = \sigma^\pi \otimes \sigma^\pi$.\\

Let $J = \bigl(\begin{array}{cc} 0 & 1\\ -1 & 0 \end{array}\bigr)$, the operator which gives $\tilde \HH$ the structure of a complex Hilbert space, and consider the one-parameter unitary group $\theta_t = \exp(\frac{\pi t}{2} J)$.  Since $\theta_t$ commutes with $\tilde \pi$, there is a well-defined one-parameter group $\alpha: \R \rightarrow \Aut(A \otimes A, \tau \otimes \tau)$ which commutes with $\tilde \sigma^\pi$ namely, 
$$
\alpha_t = \sigma_{\theta_t} = {\rm Ad}(\exp(\frac{\pi t}{2}J)^\SSS).
$$  
Let $\rho = \bigl(\begin{array}{cc} 1 & 0\\ 0 & -1 \end{array}\bigr)$, and observe $\rho \circ \theta_{-t} = \theta_t \circ \rho$.  Hence, 
$$
\beta = \sigma_\rho = {\rm Ad}(\rho^\SSS)
$$ 
conjugates $\alpha_t$ and $\alpha_{-t}$.  Finally notice that $\theta_1(\HH \oplus 0) = 0 \oplus \HH$, which gives $\alpha_1(A \otimes 1) = 1 \otimes A$.  The pair $(\alpha, \beta)$ is, thusly, an s-malleable deformation of the action $\sigma^\pi$.\\
 
Let $T \in B(\tilde \HH)$ be skew-adjoint.  Associate to $T$ the unbounded skew-adjoint operator $\partial(T)$ on $\SSS(\HH)$ defined by $$\partial(T)(\Omega) = 0,\ \partial(T)(\xi_1 \dots \xi_n) = \sum_{i=1}^n \xi_1 \dots T(\xi_i) \dots \xi_n.$$  We have that if $U(t) = \exp(t T) \in \OO(\HH)$, then $$\lim_{t \rightarrow 0}\frac{
U(t)^\SSS - I}{t} = \partial(T).$$  Let $\delta : A \otimes A \rightarrow L^2(A \otimes A)$ be the derivation defined by $$\delta(x) = [x, \partial(T)] = \lim_{t \rightarrow 0}\frac{\sigma_{U(t)}(x) - x}{t}.$$  Taking $T$ to be the operator $J$ defined above, gives us the derivation which is the infinitesimal generator of the s-malleable deformation of the Gaussian action described in this section.  From the relation $\delta( \cdot) = [\ \cdot \ , \partial(J)]$, we see that the $*$-algebra generated by the operators $s(\xi)$ forms a core for $\delta$.\\

Letting $\delta_0 = \delta|_{A \otimes 1}$, we have that $$\Phi^t = \exp(-t \delta_0^* \overline \delta_0) = \exp( -t E_{A \otimes 1} \circ \delta^* \overline \delta) = \exp(t E_{A \otimes 1} \circ \delta^2).$$  We compute $$E_{A \otimes 1} \circ \delta^2(s(\xi_1) \dots s(\xi_k)) = -k \, s(\xi_1) \dots s(\xi_k).$$  Hence, $$\Phi^t(s(\xi_1) \dots s(\xi_k)) = (1 - e^{-kt})s(\Omega) + e^{-kt} s(\xi_1) \dots s(\xi_k).$$


\section{Cohomology of Gaussian actions}
In this section, we obtain Theorem \ref{thm:1} and its corollary.  We do so by using a construction (cf. \cite{guichardet}, \cite{parthasarathyschmidt}, \cite{schmidtcohomology}) which, given an orthogonal representation and a cocycle, produces a $\T$-valued cocycles for the associated Gaussian action.  We then show that these cocycles do not untwist by applying the above deformation.\\

Let $b:\G \rightarrow \HH$ be a cocycle for an orthogonal representation $\pi:\G \rightarrow \OO(\HH)$ and $\G\car^{\sigma}(A,\tau) = (L^\infty(X,\mu), \int \cdot \, d\mu)$ be the Gaussian action associated to $\pi$, as described in section \ref{sec:gaussian}.  Viewing $\HH$ as a subset of $L_\R^2(X,\mu)$, Parthasarathy and Schmidt in \cite{parthasarathyschmidt} constructed the cocycle $c:\G \times X \rightarrow \T$ by the rule 
$$
c(\g,x) = \exp(i b(\g^{-1}))(x).
$$  
We write $\omega_\g$ for the element of $\UU(L^\infty(X,\mu))$ given by $\omega_\g(x) = c(\g,\g^{-1}x)$.  The cocycle identity for $c$ then transforms to the formula $\omega_{\g_1\g_2} = \omega_{\g_1}\sigma_{\g_1}(\omega_{\g_2})$, for all $\g_1, \g_2 \in \G$.  Moreover, $c$ is cohomologous to a homomorphism if and only if there is a unitary element $u \in \UU(L^\infty(X, \mu))$ such that
$\g \mapsto u \omega_\g \sigma_\g(u^*) $ is a homomorphism, i.e., each $u \omega_\g \sigma_\g(u^*)$ is fixed by the action of the group.

A routine calculation shows that $\tau(\omega_\g) = \int c(\g,x) d\mu(x) = \exp(-\| b(\g) \|^2/2)$.  In particular, this shows that the representation associated to the positive-definite function $\varphi(\g) = \exp(-\|b(\g)\|^2/2)$ is naturally isomorphic to the twisted Gaussian action $\omega_\g \sigma_\g$.

\begin{thm}  Using the notation above, if $\pi:\G \rightarrow \OO(\HH)$ is weak mixing, (so that $\sigma$ is ergodic) and if $b$ is an unbounded cocycle, then $c$ does not untwist.
\end{thm}

\begin{proof}
Since $\sigma$ is ergodic, if $c$ were to untwist then there would exist some $u \in \UU(A)$ such that $u \omega_\g \sigma_\g(u) \in \T$, for all $\g \in \G$.  It would then follow that any deformation
of $A$ which commutes with the action of $\G$ must converge uniformly on the set $\{ \omega_\g \ | \ \g \in \G \}$.  Indeed, this is just a consequence of the fact that completely positive deformations become asymptotically $A$-bimodular.

However, if we apply the deformation $\alpha_t$ from Section \ref{sec:infinitesimalgaussian} then we can compute
$$
\langle \alpha_{2t/\pi}(\omega_\g \otimes 1), \omega_\g \otimes 1 \rangle 
$$
$$
= \langle \exp( i (\cos t) b(\g^{-1}) ) \otimes \exp( -i (\sin t) b(\g^{-1}) ), \exp (i b(\g^{-1}) ) \otimes 1 \rangle
$$
$$
=  \exp ( (1- \cos t)^2\| b(\g) \|^2/2 + ( \sin^2 t ) \| b(\g) \|^2/2 )
$$
$$
= \exp (-(1 - \cos t) \| b(\g) \|^2)
$$

This will converge uniformly for $\g \in \G$ if and only if the cocycle $b$ is bounded and hence the result follows.
\end{proof}

\begin{cor}
The exponentiation map described above induces an injective homomorphism $H^1(\G, \pi) \to H^1(\G, \sigma, \T)/ \chi (\G)$, where $\chi( \G)$ is the character group of $\G$.
\end{cor}

\begin{proof}
It is easy to see that if two cocycles in $Z^1(\G, \pi)$ are cohomologous then the resulting cocycles for the Gaussian action will also be cohomologous.  This shows that the map described above is well defined.

The above theorem, together with the fact that this map is a homomorphism, shows that this map is injective.
\end{proof}

Since a nonamenable group has vanishing first $\ell^2$-Betti number if and only if it has vanishing first cohomology into its left regular representation \cite{bekkavalette}, \cite{petersonthom}, we derive the following corollary.

\begin{cor}
Let $\G$ be a nonamenable countable discrete group, and let $\G\car^\sigma(X,\mu)$ be the Bernoulli shift action.
If $\beta_1^{(2)}(\G) \neq 0$ then $H^1(\G, \sigma, \T) \neq \chi(\G)$, where $\chi(\G)$ is the group of characters.  In particular, $\G \car^\sigma (X,\mu)$ is not $\UU_{\rm fin}$-cocycle superrigid.
\end{cor}


\section{Derivations}\label{sec:derivations}

In this section we continue our investigation of deformations, but this time on the infinitesimal level.  

\subsection{Derivations from s-malleable deformations}

Let $(M, \tau)$ be a finite von Neumann algebra, and let $\alpha : \R \rightarrow \Aut(M, \tau)$ be a point-wise strongly continuous one-parameter group of automorphisms. Let $\delta$ be the infinitesimal generator of $\alpha$, \textit{i.e.}, $\exp(t \delta) = \alpha_t$. For $f \in L^1(\R)$ define the bounded operator $\alpha_f : M \rightarrow M$ by $$\alpha_f(x) = \int_{-\infty}^\infty f(s)\alpha_s(x)ds.$$  It can be checked that if $f \in C^1(\R) \cap L^1(\R)$ and $f' \in L^1(\R)$, then $$\delta \circ \alpha_f(x) = -\alpha_{f'}(x).$$  Also if $x\in M \cap D(\delta)$, then we have that $$\alpha_t(x) - x = \int_0^t \delta \circ \alpha_s(x)ds = \int_0^t \alpha_s(\delta(x))ds.$$

\begin{thm} Suppose that for every $\e > 0$, there exists $f \in C^1(\R) \cap L^1(\R)$ such that $f' \in L^1(\R)$ and $\sup_{x \in (M)_1} \|\alpha_f(x) - x\|_2 \leq \e/4$.  Then $\alpha_t$ converges $\| \cdot\|_2$-uniformly to the identity on $(M)_1$ as $t \rightarrow 0$.
\end{thm}

\begin{proof} We need only show for every $\e > 0$ that there exists some $\eta > 0$ such that for all $t<\eta$, $\sup_{x \in (M)_1} \|\alpha_t(x) - x\|_2 \leq \e $. Let $\tilde x = \alpha_f(x)$. We have that $\|\alpha_t(x) - x\|_2 \leq \|\alpha_t (\tilde x) - \tilde x\|_2 + \e/2$.  Since $\delta \circ \alpha_f$ is defined everywhere, $\delta \circ \alpha_f: M \rightarrow L^2(M,\tau)$ is bounded. In fact, $\|\delta \circ \alpha_f\| \leq \|f'\|_{L^1}$. Now, since $\tilde x \in D(\delta)$, we have $\alpha_t(\tilde x) - \tilde x = \int_0^t \alpha_s(\delta(\tilde x))ds$.  Hence $\|\alpha_t(\tilde x) - \tilde x\|_2 \leq t\|f'\|_{L^1}$. Choosing $\eta = \e(2\|f'\|_{L^1})^{-1}$ does the job.
\end{proof}

\begin{cor}\label{cor:convergence}  If $\varphi_t = \exp(-t\delta^*\delta)$ converges uniformly to the identity as $t \rightarrow 0$, then so does $\alpha_t$.
\end{cor}

\begin{proof} Let $f_t(s) = \frac{1}{\sqrt{4 \pi t}}\exp(-s^2/4t)$; then, $\varphi_t(x) = \int_{-\infty}^\infty f_t(s)\alpha_s(x)ds$ follows by completing the square.  
\end{proof}


\subsection{Tensor products of derivations}\label{sec:tensorproduct}
We describe here the notion of a tensor product of derivations; see also Section 6 of \cite{petersonl2}.   

Consider $N_i$, $i \in I$ a family of finite von Neumann algebras with normal faithful traces $\tau_i$.  If $\delta_i:N_i \rightarrow \HH_i$ is a family of closable 
real derivations into Hilbert bimodules $\HH_i$ with domains $D(\delta_i)$ then we may consider the dense $*$-subalgebra 
$D(\delta) = \otimes_{i \in I}^{\rm alg} D(\delta_i) \subset N = \overline \otimes_{i \in I} N_i$.

We denote by $\hat{N_j}$ the tensor product of the $N_i$'s obtained by omitting the $j$th index so that we have a natural identification 
$N = \hat{N_j} \overline \otimes N_j$ for each $j \in I$.  Let $\HH = \bigoplus_{j \in I} \HH_j \overline \otimes L^2(\hat{N_j})$ which is naturally 
a Hilbert bimodule because of the identification $N = \hat{N_j} \overline \otimes N_j$.

The tensor product of the derivations $\delta_i$, $i \in I$ is defined to be the derivation $\delta = \bigotimes_{i \in I} \delta_i : D(\delta) \rightarrow \HH$
which satisfies 
$$
\delta( \otimes_{i \in I} x_i ) = \bigoplus_{j \in I} ( \delta_j(x_j) \otimes_{i \in I, i \not= j} x_i ).
$$
This is well defined as only finitely many of the
$x_i$'s are not equal to $1$ and hence the right hand side is a finite sum.

If $\Phi_i^t = \exp(-t\delta_i^*\overline{\delta_i})$ is the semigroup deformation associated to $\delta_i$ then one easily checks that the semigroup 
deformation associated to $\delta$ is $\Phi^t = \bigotimes_{i \in I} \Phi_i^t : N \rightarrow N$.  A similar formula holds for the resolvent deformation.
Note that by viewing the Hilbert bimodule associated with $\Phi^t$ and using the usual ``averaging trick'' (e.g. Theorem 4.2 in \cite{popacorrespondence}) 
it follows that $\Phi^t$ will converge uniformly in $\| \cdot \|_2$ to the identity on $(N)_1$ if and only if each $\Phi_t^i$ converges uniformly in $\| \cdot \|_2$
to the identity on $(N_i)_1$ and moreover this convergence is uniform in $i \in I$.

%

\subsection{Derivations from generalized Bernoulli shifts}

We use here the notation in Section \ref{sec:gaussian} above.  Given a real Hilbert space $\HH$, we consider the new Hilbert space
$\HH' = \R \Omega_0 \oplus \HH$.  If $\xi \in \HH$ is a non-zero element we denote by $P_\xi$ the rank one projection onto $\xi$.
We denote by $\tilde\HH$ the tensor product (complex) Hilbert space  
$\HH \otimes \SSS(\HH')$

Let $N \in \N \cup \{ \infty \}$ be the dimension of $\HH$ and consider an orthonormal basis 
$\beta = \{ \xi_n \}_{i =1}^N$ for $\HH$.  We then define a left action of $A$, the von Neumann
algebra generated by the spectral projections of $s(\xi)$, $\xi \in \HH$, on $\tilde\HH$ such that for each $\xi \in \HH$, $s(\xi)$ acts on the left (as an unbounded operator) by
$$
\ell_\beta(s(\xi)) = {\rm id} \otimes s(\xi).
$$
We also define a right action of $A$ on $\tilde\HH$ such that for each $\xi \in \HH$, $s(\xi)$ acts on the right by extending linearly the formula
\begin{equation}\label{eq:bimodulestructure}
r_\beta(s(\xi)) (\xi_n \otimes \eta) = P_{\xi_n}(\xi) \otimes S(\Omega_0)\eta + \xi_n \otimes s(\xi - P_{\xi_n}(\xi)) \eta, 
\end{equation}
for each $1 \leq n \leq N, \eta \in \SSS(\HH')$.

These formulas define unbounded self-adjoint operators on $\tilde\HH$ in general; however,
by functional calculus they extend to give commuting normal actions of $A$ on $\tilde\HH$.

Moreover, if $T \in \OO(\HH) \subset \OO(\HH')$, then we have that for any $\xi \in \HH$
$$
\ell_{T\beta}( s(T\xi) ) = \ell_{T\beta}(\sigma_T(s(\xi))) = {\rm Ad}(T \otimes T^\SSS) \ell_\beta(s(\xi)).
$$
Also,
$$
r_{T\beta}(s(T\xi) ) = r_{T\beta}(\sigma_T(s(\xi))) = {\rm Ad}(T \otimes T^\SSS) (r_\beta(s(\xi))).
$$
From here on we will denote the left action of $A$ on $\tilde\HH$ by $\ell_\beta(a) x = a \cdot_\beta x$
and the right action by $r_\beta(a) x = x \cdot_\beta a$.
By extending the formulas above to $A$ we have the following lemma.

\begin{lem}\label{lem:consistency}
Using the notation above, consider the inclusion $\OO(\HH) \subset \UU(\tilde\HH)$ given by 
$T \mapsto \tilde T = T \otimes T^\SSS$.
Then for each $T \in \OO(\HH)$, $x, y \in A$, and $\tilde\xi \in \tilde\HH$, we have
$\tilde T( x \cdot_\beta \tilde\xi \cdot_\beta y) = \sigma_T(x) \cdot_{T\beta} ( \tilde T\tilde\xi)  \cdot_{T\beta} \sigma_T(y)$.
\end{lem}

\begin{rem}
While we will not use this in the sequel, an alternate way to view the $A$-$A$ Hilbert bimodule structure on $\tilde\HH$ is as follows.
Given our basis $\beta = \{ \xi_n \}_{n=1}^N \subset \HH$,
consider the probability space $(X, \mu) = \Pi_n (\R, g)$ where $g$ is the Gaussian measure on $\R$.  We can identify 
$A = L^\infty(X, \mu)$, and we denote by $\pi_n \in L^2(X, \mu)$ the projection onto the $n$th copy of $(\R, g)$ so that 
the $\pi_n$'s are I.I.D. Gaussian random variables.  

We embed $\HH$ into $L^2(X, \mu)$ linearly  by the map $\eta$ such that $\eta(\xi_n) = \pi_n$ given an orthogonal 
transformation $T \in \OO(\HH)$, we associate to $T$ the unique measure-preserving automorphism $\sigma_T \in {\rm Aut}(A)$
such that $\sigma_T( \eta(\xi) ) = \eta( T\xi )$, for all $\xi \in \HH$.

For each $k$ we denote by 
$$
A_k = ( \bigotimes_{n < k} L^\infty(\R, g)) \otimes ( L^\infty(\R, g) \otimes L^\infty(\R, g) ) \otimes ( \bigotimes_{n > k} L^\infty(\R, g) ),
$$
and we view $L^2(A_k)$ as an $A$-$A$ bimodule so that 
$$
(\otimes_{n} a_n) \cdot x = ( \bigotimes_{n < k} a_n \otimes ( a_k \otimes 1 ) \otimes  \bigotimes_{n > k} a_n ) x
$$
and 
$$
x \cdot (\otimes_{n} a_n) = x( \bigotimes_{n < k} a_n \otimes ( 1 \otimes a_k ) \otimes  \bigotimes_{n > k} a_n ),
$$ 
for $x \in L^2(A_k)$.

Consider the $A$-$A$ Hilbert bimodule $\bigoplus_k L^2(A_k)$, and note that it is canonically identified with the Hilbert space 
$\HH \otimes L^2(A_1)\cong \HH \otimes L^2(\R, g) \otimes L^2(A) \cong \tilde\HH$ in a way which preserves the $A$-$A$ bimodule structure.  
Under this identification the inclusion $\OO(\HH) \subset \UU( \HH \otimes L^2(\R, g) \otimes L^2(A))$ becomes $T \mapsto T \otimes {\rm id} \otimes \sigma_T$.
\end{rem}

We now consider the algebra $A_0 \subset L^2(A)$ of square summable operators generated by $s(\xi)$, $\xi \in \HH$, and define
a derivation $\delta_\beta$ on $A_0$ by setting 
$$
\delta_\beta(s(\xi)) = \xi \otimes \Omega \in \tilde\HH,
$$ for each $\xi \in \HH$.
Note that the formula for $\delta_\beta(s(\xi))$ does not depend on the basis $\beta$, but the bimodule structure that we are imposing
on $\HH$ does depend on $\beta$.
If $\xi_0, \xi_1, \ldots, \xi_k \in \beta$ such that $\xi_0$ is orthogonal to the vectors $\xi_1, \ldots, \xi_k$ then it follows that
$\delta_\beta(s(\xi_1) \cdots s(\xi_k))$ is a $s(\xi_0)$-central vector and hence by induction on $k$ it follows that $\delta_\beta$ is well defined.
Also, since $\delta_\beta$ extends to a bounded operator on $\overline{\rm sp} \{ s(\xi_1) \cdots s(\xi_k) \ | \ \xi_1, \ldots, \xi_k \in \HH \}$ for each $k$ it follows 
that $\delta_\beta$ is a closable operator and if we still denote by $\delta_\beta$ the closure of this operator we have that
$x \mapsto \| \delta_\beta(x) \|^2$ is a quantum Dirichlet form on $L^2(A)$ (see \cite{davieslindsay, sauvageot1, sauvageot2}).

In particular, it follows from \cite{davieslindsay} that $D(\delta_\beta) \cap A$ is a weakly dense $*$-subalgebra and 
${\delta_\beta}_{|D(\delta_\beta) \cap A}$ is a derivation.

Note that if we identify $\tilde\HH$ with $\bigoplus_k L^2(A_k)$ as above then $\delta_\beta$
can also be viewed as the tensor product derivation 
$\delta_\beta = \bigotimes_k \delta_k$ where $\delta_k : L^2(\R, g) \rightarrow L^2(\R, g) \overline \otimes L^2(\R, g)$
is the difference quotient derivation for each $k$, \textit{i.e.}, $\delta_k(f)(x, y) = \frac{f(x) - f(y)}{x - y}$.  

\begin{lem}\label{lem:invariant}
Using the above notation, $\delta_\beta$ is a densely defined closed real derivation, 
$s(\HH) \subset D(\delta_\beta)$, $\delta_\beta \circ s: \HH \rightarrow \tilde\HH$ is an isometry,
and for all $T \in \OO(\HH)$, $\sigma_T(D(\delta_\beta)) = D(\delta_{T\beta})$, and
$\delta_{T\beta}(\sigma_T(a)) = \tilde T(\delta_\beta(a) )$, for all $a \in D(\delta_\beta)$.
\end{lem}

\begin{proof}
The fact that $s(\HH) \subset D(\delta_\beta)$, and that $\delta_\beta \circ s$ is an isometry follows from 
the formula $\delta_\beta(s(\xi)) = \xi \otimes \Omega$ above.  

Moreover for $\xi \in \HH$ we have 
$$
\delta_{T\beta}( \sigma_T(s(\xi))) = T\xi \otimes \Omega
$$
$$
= (T \otimes T^\SSS)( \xi \otimes \Omega ) = \tilde T \delta_\beta ( s(\xi) ).
$$
By Lemma \ref{lem:consistency} this formula then extends to $A_0$, and since $\tilde T$ acts on $\tilde\HH$ unitarily 
and $A_0$ is a core for $\delta_\beta$ we have that 
$\sigma_T(D(\delta_\beta)) = D(\delta_{T\beta})$ and this formula remains valid for $a \in D(\delta_\beta)$.
\end{proof}

Given an action of a countable discrete group $\G$ on a countable set $S$ we may consider the generalized 
Bernoulli shift action of $\G$ on $(X, \mu) = \Pi_{s \in S} (\R, g)$ given by $\g (r_s)_{s \in S} = (r_{\g^{-1}s})_{s \in S}$.
If we set $\HH = \ell^2S$ and consider the corresponding representation $\pi:\G \rightarrow \UU(\HH)$ 
then the generalized Bernoulli shift can be viewed as the Gaussian action corresponding to $\pi$.
Moreover we have that the canonical basis $\beta = \{ \delta_s \}_{s \in S}$ is invariant to 
the representation, \textit{i.e.}, $\pi_\g\beta = \beta$, for all $\g \in \G$.

In this case by Lemma \ref{lem:invariant} we have that $D(\delta_\beta)$ is $\sigma_\g$ invariant for all $\g \in \G$ 
and $\delta_{\beta}(\sigma_\g(a)) = \tilde\pi_\g( \delta_\beta(a) )$, for all $\g \in \G$, $a \in D(\delta_\beta)$, 
where $\tilde\pi:\G \rightarrow \UU(\tilde\HH)$ is the unitary representation given by $\tilde\pi = \pi \otimes \pi^\SSS$.
If we denote by $N = A \rtimes \G$ the corresponding group-measure space construction
then using Lemma \ref{lem:consistency} we may define an $N$-$N$ Hilbert bimodule structure on $\KK = \tilde\HH \otimes \ell^2\G$ which satisfies
$$
(a u_{\g_1} ) (\xi \otimes \delta_{\g_0}) (b u_{\g_2}) = (a \cdot_\beta (\tilde\pi_{\g_1}\xi) \cdot_\beta \sigma_{\g_1\g_0}(b) ) \otimes \delta{\g_1\g_0\g_2},
$$
for all $a, b \in A$, $\g_0, \g_1, \g_2 \in \G$, and $\xi \in \tilde\HH$.
We may then extend $\delta_\beta$ to a closable derivation
$\delta: \ast$-${\rm Alg}(D(\delta_\beta) \cap A, \G) \rightarrow \KK$ such that
$\delta(au_\g) = \delta_{\beta}(a) \otimes u_\g$, for all $a \in D(\delta_\beta)$, $\g \in \G$.

As above we denote by  $\zeta_\alpha: N \rightarrow N$ the unital, symmetric, c.p. resolvent maps given by
$\zeta_\alpha = ( \alpha / (\alpha + \delta^*\overline\delta) )^{1/2}$, for $\alpha > 0$.

Note that if $M$ is a finite von Neumann algebra then we let $\G$ act on $M$ trivially and we may extend the derivation $\delta$ to 
$(A \overline \otimes M) \rtimes \G \cong (A \rtimes \G) \overline \otimes M$ by considering the tensor product derivation of $\delta$ 
with the trivial derivation (identically $0$) on $M$.  In this case the corresponding deformation of resolvent maps is just $\zeta_\alpha \otimes {\rm id}$.

\begin{lem}\label{lem:deformation}
Consider Ioana's deformation $\alpha_t$ on $A$ corresponding to generalized Bernoulli shift as described above in Section \ref{sec:chifanioana}.
If $M$ is a finite von Neumann algebra and $B \subset (A \overline\otimes M) \rtimes \G$ is a subalgebra such that 
$\zeta_\alpha$ converges uniformly to the identity on $(B)_1$ as $\alpha \rightarrow 0$
then $\alpha_t$ converges uniformly to the identity on $(B)_1$ as $t \rightarrow 0$.
\end{lem}

\begin{proof}
The infinitesimal generator of Ioana's deformation cannot be identified with $\delta$ as the $\alpha_t$'s will converge uniformly on 
the algebra generated by $s(\xi)$ for each $\xi \in \beta$, and $\zeta_\alpha$ will not have this property.  However, it is not hard to check using the fact that
both derivations arise as tensor product derivations that
if $\zeta_\alpha^0$ are the resolvent maps corresponding to the infinitesimal generator of $\alpha_t$ then we have the inequality
$\tau(\zeta_\alpha(a)a^*) \leq 2\tau(\zeta_\alpha^0(a)a^*)$, for all $a \in A$.  
Hence the lemma follows from Lemma 2.1 in \cite{petersonl2} and Corollary \ref{cor:convergence} above.
\end{proof}

\begin{rem}
It can be shown in fact that the deformation coming from the derivation above, 
Ioana's deformation, and the s-malleable deformation from the Gaussian action, are successively weaker deformations.
That is to say that one deformation converging uniformly on a subset of the unit ball implies that the next deformation must also 
converge uniformly.
\end{rem}

When we restrict the bimodule structure on $\KK$ to the subalgebra $L\G$ we see that this is exactly the bimodule
structure coming from the representation $\tilde\pi = \pi \otimes \pi^\SSS$, this give rise to the following lemma:

\begin{lem}\label{lem:repbimodule}
Using the notation above, given $H < \G$ we have the following:

$1$.  ${}_{LH}\KK_{LH}$ embeds into a direct sum of coarse bimodules if and only if $\pi_{|H}$ embeds into a direct sum of left regular representations.

$2$.   ${}_{LH}\KK_{LH}$ weakly embeds into a direct sum of coarse bimodules if and only if $\pi_{|H}$ weakly embeds into a direct sum of left regular representations.

$3$.   ${}_{LH}\KK_{LH}$ has stable spectral gap if and only if $\pi_{|H}$ has stable spectral gap.

$4$.   ${}_{LH}\KK_{LH}$ is a compact correspondence if and only if $\pi_{|H}$ is a $c_0$-representation.

$5$.   ${}_{LH}\KK_{LH}$ is weakly mixing if and only if $\pi_{|H}$ is weakly mixing. 

\end{lem}


\section{$L^2$-rigidity and $\UU_{\mathrm {fin}}$-cocycle superrigidity}

In this section we use the tools developed above to prove $\UU_{\mathrm {fin}}$-cocycle superrigidity for the Bernoulli shift action which we view as 
the Gaussian action corresponding to the left-regular representation.  

To prove that a cocycle untwists we use the same general 
setup as Popa in \cite{popasuperrigid}.  In particular, we use the fact that for a weakly-mixing action, in order to show 
that a cocycle untwists it is enough to show that the corresponding s-malleable deformation converges uniformly on 
the ``twisted'' subalgebra of the crossed product algebra.  The main difference in our approach is that to 
show that the s-malleable deformation converges uniformly it is enough by Lemma \ref{lem:deformation} to show that
the deformation coming from the Bernoulli shift derivation converges uniformly.  This allows us to use the techniques developed
in \cite{petersonT}, \cite{petersonl2}, \cite{ozawapopaII}, and \cite{petersonwe} to analyze the 
cocycle on the level of the base space itself rather than the exponential of the space where the properties can be somewhat hidden.

\begin{thm}\label{thm:l2bernoulli}
Let $\G$ be a countable discrete group.  If $L\G$ is $L^2$-rigid then the Bernoulli shift action with diffuse core of $\G$ is $\UU_{\mathrm {fin}}$-cocycle superrigid.
\end{thm}

\begin{proof}
Let $\GG \in \UU_{\mathrm {fin}}$, then $\GG \subset \UU(M)$ as a closed subgroup where $M$ is a finite separable von Neumann algebra.
Let $c:\G \times X \rightarrow \GG$ be a cocycle where $X$ is the probability space of the Gaussian action.
Consider $A = L^\infty(X)$, and $\omega: \G \rightarrow \UU(A \overline \otimes M)$ 
given by $\omega_\g(x) = c(\g, \g^{-1}x)$ the corresponding unitary cocycle for the action $\tilde\sigma_\g = \sigma_\g \otimes {\rm id}$.
Note that $\omega_{\g_1\g_2} = \omega_{\g_1} \tilde\sigma_{\g_1}(\omega_{\g_2})$, for all $\g_1, \g_2 \in \G$.
Here we view a unitary element in $A \overline \otimes M$ as map from $X$ to $\UU(M)$ (see \cite{popasuperrigid} for a detailed explanation).

As noted above, the Bernoulli shift action with diffuse core is precisely the Gaussian action corresponding to the left-regular representation; 
hence, by Lemma \ref{lem:repbimodule} we have that as an $L\G$-$L\G$ Hilbert bimodule $\KK$ embeds into a direct sum of coarse correspondences.
If we denote by $\widetilde{L\G}$ the von Neumann algebra generated by $\{ \tilde u_\g \} = \{ \omega_\g u_\g \}$ then the bimodule structure
of $\widetilde{L\G}$ ($\cong L\G$) on $\KK$ is the same as the bimodule structure of $L\G$ on the correspondence coming from the representation
$\g \mapsto {\rm Ad}(\omega_\g) \circ \tilde\pi_\g$ on $\tilde\HH \otimes L^2M$.  
The $A \overline \otimes M$ bimodule structure on $\tilde\HH \otimes L^2M = \HH \otimes \SSS(\HH') \otimes L^2M$ decomposes
as a direct sum of bimodules $\HH \otimes \SSS(\HH') \otimes L^2M = \oplus_{\xi \in \beta} \SSS(\HH') \otimes L^2M$ where the bimodule structure
on each copy of $\SSS(\HH') \otimes L^2M$ is given by Equation (\ref{eq:bimodulestructure}), and under this decomposition
we have ${\rm Ad}(\omega_\g)  \circ \tilde\pi_\g = \pi_\g \otimes ({\rm Ad}(\omega_\g) \circ \pi^\SSS_\g )$.
Therefore by Fell's absorption principle this representation is an infinite direct sum
of left-regular representations; hence, we have that $\KK$ also embeds into a direct sum of coarse correspondences when $\KK$ is viewed as an $\widetilde{L\G}$-$\widetilde{L\G}$ Hilbert bimodule.

Since $L\G$ is $L^2$-rigid we have that the corresponding deformation $\zeta_\alpha$ converges uniformly to the identity map 
on $(\widetilde{L\G})_1$, by Lemma \ref{lem:deformation} we have that a corresponding s-malleable deformation 
also converges uniformly to the identity on $(\widetilde{L\G})_1$.  Thus, by Theorem 3.2 in \cite{popasuperrigid} the cocycle $\omega$ is cohomologous to a homomorphism.
\end{proof}

We end this paper with some examples of groups for which the hypothesis of the Theorem \ref{thm:l2bernoulli} is satisfied.

It follows from \cite{petersonl2} that if $N$ is a nonamenable II$_1$ factor which is non-prime, has property ($\Gamma$), or is $w$-rigid,
then $N$ is $L^2$-rigid.  We include here another class of $L^2$-rigid finite von Neumann algebras, this class
includes the group von Neumann algebras of 
all generalized wreath product groups $A_0 \wr_X \G_0$ where $A_0$ is an infinite abelian group and $\G_0$ does not have the Haagerup property, or $\G_0$ is a non-amenable direct products of infinite groups.  
This is a special case of a more general result which can be found in \cite{petersonwe}.

\begin{thm}
Let $\G$ be a countable discrete group which contains an infinite normal abelian subgroup and either does not have the Haagerup property or contains an infinite subgroup $\G_0$ such that $L\G_0$ is $L^2$-rigid,
then $L\G$ is $L^2$-rigid.
\end{thm}

\begin{proof}
We will use the same notation as in \cite{petersonl2}.
Suppose $(M, \tau)$ is a finite von Neumann algebra with $L\G \subset M$, and $\delta: M \rightarrow L^2M \overline \otimes L^2M$ 
is a densely defined closable real derivation.  

Since the maps $\eta_\alpha$ converge point-wise to the identity we may take an appropriate sequence $\alpha_n$ such that the 
map $\phi: \G \rightarrow \R$ given by $\phi(\gamma) = \Sigma_n 1 - \tau( \eta_{\alpha_n}(u_\gamma) u_\gamma^*)$ is well defined.
If the deformation $\eta_\alpha$ does not converge uniformly on any infinite subset of $\G$ then the map $\phi$ is not bounded 
on any infinite subset and hence defines a proper, conditionally negative-definite function on $\G$ showing that $\G$ has the
Haagerup property.

Therefore if $\G$ does not have the Haagerup property then there must exist an infinite set $X \subset \G$ on which the deformation $\eta_\alpha$ converges uniformly.  Similarly, if $\G_0 \subset \G$ is an infinite subgroup such that $L\G_0$ is $L^2$-rigid then we have that
the deformation $\eta_\alpha$ converges uniformly on the infinite set $X = \G_0$.

Let $A \subset \G$ be an infinite normal abelian subgroup.  If there exists an $a \in A$ such that
$a^X = \{ xax^{-1} | x \in X \}$ is infinite, then we have that the deformation $\eta_\alpha$ converges uniformly 
on this set, and by applying the results in \cite{petersonl2} it follows that $\eta_\alpha$ converges uniformly 
on $A \subset L(A)$.  Since $A$ is a subgroup in $\UU(LA)$ which generates $LA$ it then follows 
that $\eta_\alpha$ converges uniformly on $(LA)_1$
and hence also on $(L\G)_1$ since $A$ is normal in $\G$.

If $a \in A$ and $a^X$ is finite then there exists an infinite sequence $\gamma_n \in X^{-1}X$ such that
$[\gamma_n, x] = e$, for each $n$.  Thus if $a^X$ is finite for each $a \in A$ then by taking a diagonal subsequence
we construct a new sequence $\gamma_n \in X^{-1}X$ such that $\lim_{n \rightarrow \infty} [\gamma_n, a ] = e$.
Since $\eta_\alpha$ also converges uniformly on $X^{-1}X$ we may again apply the results in \cite{petersonl2}
to conclude that $\eta_\alpha$ converges uniformly on $A$ and hence on $(L\G)_1$ as above.
\end{proof}

It has been pointed out to us by Adrian Ioana that in light of Corollary 1.3 in \cite{chifanioanaII} the above argument is sufficient to show that for a lattice $\G$ in a connected Lie group which does not have the Haagerup property, we must have that $L\G$ is $L^2$-rigid.

We also show that $L^2$-rigidity is stable under orbit equivalence.  The proof of this uses the diagonal embedding argument of Popa and Vaes \cite{popavaesII}.

\begin{thm}\label{thm:l2rigidityoe}
Let $\G_i \curvearrowright (X_i, \mu_i)$ be free ergodic measure preserving actions for $i = 1,2$.  If the two actions are orbit equivalent and $L\G_1$ is $L^2$-rigid then
$L\G_2$ is also $L^2$-rigid.
\end{thm}

\begin{proof}
Suppose $L\G_2 \subset M$ and $\delta: M \rightarrow \HH$ is a closable real derivation such that $\HH$ as an $L\G_2$ bimodule 
embeds into a direct sum of coarse bimodules.  Let $N = L^\infty(X_1, \mu_1) \rtimes L\G_1 = L^\infty(X_2, \mu_2) \rtimes L\G_2$ and consider the
$N \overline \otimes M$ bimodule $\tilde \HH = L^2N \overline \otimes \HH$.
If we embed $N$ into $N \overline \otimes M$ by the linear map 
$\alpha$ which satisfies $\alpha( a u_\g) = a u_\g \otimes u_\g$ for all $a \in L^\infty(X_2, \mu_2)$, and $\g \in \G_2$, then
when we consider the $\alpha(N)$-$\alpha(N)$ bimodule $\tilde \HH$ we see that this bimodule is contained in a direct sum of the bimodule $L^2\langle \alpha(N), \alpha(L^\infty(X_1, \mu_1)) \rangle$ coming from the basic construction of $(\alpha(L^\infty(X, \mu)) \subset \alpha(N) )$.  Indeed, this follows because the completely positive maps corresponding to left and right bounded vectors of the form $1 \otimes \xi \in L^2N \overline \otimes \HH$ are easily seen to live in $L^2\langle \alpha(N), \alpha(L^\infty(X_1, \mu_1)) \rangle$.

The $\alpha(N)$-$\alpha(N)$ bimodule $L^2\langle \alpha(N), \alpha(L^\infty(X_1, \mu_1)) \rangle$ is an orbit equivalence invariant and is canonically isomorphic to the bimodule coming from the left regular representation of $\G_1$ (see for example Section 1.1.4 in \cite{popabetti}).
It therefore follows that $\tilde \HH$ when viewed as an $\alpha( L\G_1)$ bimodule embeds into a direct sum of coarse bimodules.

We consider the closable derivation $0 \otimes \delta: N \overline \otimes M \rightarrow \tilde \HH$ as defined in Section \ref{sec:tensorproduct} 
and use the fact that $L\G_1$ is $L^2$-rigid to conclude that the corresponding deformation ${\rm id} \otimes \eta_\alpha$ converges uniformly
on the unit ball of $\alpha(N)$, (note that ${\rm id} \otimes \eta_\alpha$ is the identity on $\alpha( L^\infty(X_1, \mu_1) ) = \alpha (L^\infty(X_2, \mu_2) )$).
In particular, ${\rm id} \otimes \eta_\alpha$ converges uniformly on $\{ \alpha(u_\g) \ | \ \g \in \G_2 \}$ which shows that
$\eta_\alpha$ converges uniformly on $\{ u_\g \ | \ \g \in \G_2 \}$.  As this is a group which generates $L\G_2$ we may then use a standard averaging argument to conclude that $\eta_\alpha$ converges uniformly on the unit ball of $L\G_2$, (see for example Theorem 4.1.7  in \cite{popacorrespondence}).
\end{proof}

\begin{rem}
The above argument will also work to show that the ``$L^2$-Haagerup property'' (see \cite{petersonl2}) is preserved by orbit equivalence.  In particular,  this gives a new way to show that the von Neumann algebras of groups which are orbit equivalent to free groups are solid in the sense of Ozawa \cite{ozawasolid}.  Solidity of group von Neumann algebras for groups which are orbit equivalent to free groups was first shown by Sako in \cite{sako}.  

We also note that by \cite{cowlingzimmer} any group which is orbit equivalent to a free group will have the complete metric approximation property.  It will no doubt follow by using the techniques in \cite{ozawapopaII} that the von Neumann algebra of a group which is orbit equivalent to a free group will be strongly solid.

Examples of groups which are orbit equivalent to a free group can be found in \cite{gaboriaufree}, and \cite{bridsontweedalewilton}.
\end{rem}


\end{document}